\documentclass[reqno]{amsart}

\usepackage[utf8]{inputenc}
\usepackage[T1]{fontenc}
\usepackage{amsmath,amsfonts,amssymb}
\usepackage{color}
\usepackage{textcmds}
\usepackage{subcaption}
\usepackage[european]{circuitikz}
\usepackage{pdflscape}
\usepackage{hyperref}

\usepackage{multirow}
\usepackage{multicol}
\usepackage{bm}
\usepackage{enumitem}
\usepackage{graphicx} 
\usepackage{booktabs} 
\usepackage{makecell} 
\usepackage{mathtools}
\usepackage{csquotes}
\usepackage{tabularx}
\usepackage{nicefrac}
\usepackage{bbm}
\usepackage{mathrsfs}
\usepackage{pgfplots}
\pgfplotsset{compat=1.18} 
\usepgfplotslibrary{groupplots} 

\theoremstyle{plain}
\newtheorem{theorem}{Theorem}
\newtheorem{lemma}[theorem]{Lemma}
\newtheorem{definition}[theorem]{Definition}
\newtheorem{remark}[theorem]{Remark}
\newtheorem{example}[theorem]{Example}

\begin{document}

\title[Fourth-Order Energy-consistent AVF-DGM for Dissipative Hamiltonian Systems]{Fourth-Order Energy-Consistent Average Vector Field  \\ Discrete Gradient Methods for \\ Dissipative Hamiltonian Systems}

\author[Lautwein et al.]{Lucas Lautwein$^{1}$ \and Nicole Marheineke$^{1,\star}$ \and H{\r{a}}kon Noren Myhr$^{2}$ \and Kevin Schäfers$^{3}$}

\date{\today\\
$^1$ Trier University, Chair Modeling and Numerics, D-54286 Trier, Germany\\
$^2$ Norwegian University of Science and Technology, 7491 Trondheim, Norway \\
$^3$ University of Wuppertal, School of Mathematics and Natural Sciences, Gaußstr.~20, D-42119 Wuppertal, Germany\\
$^\star$ corresponding author, e-mail: marheineke@uni-trier.der, orcid: 0000-0002-5912-3465}

\begin{abstract}
Discrete gradient methods provide a structure-preserving approach to energy-con\-sistent time integration of dissipative Hamiltonian systems. In their conventional formulation, however, they generally attain at most second-order convergence. Extending the perturbation principle previously developed for Poisson systems to the dissipative setting, we derive fourth-order average vector field discrete gradient schemes. A congruence-based completion preserves the positive semi-definiteness of the dissipation matrix, thereby guaranteeing unconditional energy dissipativity for all positive step sizes. This completion concept provides a systematic route to higher-order energy-consistent schemes. For purely dissipative systems, an alternative exponential completion is derived, together with structure-preserving rational approximations. In particular, we characterize Pad\'e approximants that preserve positive semi-definiteness and can be incorporated into energy-associated splitting schemes. Numerical experiments on a damped physical pendulum and a dissipative Fermi--Pasta--Ulam system demonstrate the importance of the positive semi-definite completion for avoiding nonphysical numerical instabilities and capturing the correct qualitative dynamics. They further show that the proposed methods achieve the desired accuracy and energy consistency at a computational cost competitive with existing Galerkin-type approaches. Moreover, the fourth-order schemes outperform the conventional second-order discrete gradient method in computational efficiency over the considered accuracy range, despite their higher per-step cost.
\end{abstract}

\keywords{Discrete gradient methods; dissipative Hamiltonian systems; energy consistency; structure-preserving numerical methods; Pad\'{e} approximation\\
\textit{MSC.} 65P10; 65L05; 37M15\\
\textit{Acknowledgements.}
This work was partially funded by the Deutsche Forschungsgemeinschaft (DFG, German Research Foundation) -- Project-ID 531152215 -- CRC 1701.
The authors thank S{\o}lve Eidnes for helpful discussions during the preparation of this manuscript. \\
\textit{Declaration of the use of AI.}
The authors used generative AI tools, specifically ChatGPT5.6 Sol and Claude Fable~5, to formalize the idea of a structure-preserving higher-order correction in terms of the congruence-based completion. The authors reviewed and edited the generated suggestions and take full responsibility for the content of this article.}

\maketitle

\section{Introduction}
\label{sec:introduction}

This work concerns the numerical time integration of dissipative Hamiltonian systems. A characteristic property of such systems is the non-increasing behavior of the Hamiltonian along solution trajectories. We consider time-invariant systems
\begin{equation} \label{eq:system}
    \dot{x} = [J-R] \, \nabla \mathcal{H}(x), \quad x(0)=x_0, \qquad \qquad J=-J^\top, \quad R=R^\top \succeq 0,
\end{equation}
for $t\in [0,T]$, whose dynamics satisfy the dissipation equality 
\begin{equation} \label{eq:dissipation}
    \tfrac{\mathrm{d}}{\mathrm{d}t}\mathcal{H}(x(t))= -\nabla \mathcal{H}(x(t))^\top \, R  \, \nabla \mathcal{H}(x(t)) \leq 0.
   \end{equation}
Preserving this energy-dissipation property at the discrete level is important for reliable simulations. We particularly aim at higher-order energy-consistent numerical methods in the spirit of \cite{moench2026a} that reproduce the dissipative behavior of general Hamiltonians $\mathcal{H}$ and recover the energy-conserving identity for $R=0$ exactly, without imposing any restriction on the time step. 

For quadratic Hamiltonians, high-order energy-consistent schemes are well established through Gauss collocation and symplectic Runge--Kutta methods, see, e.g., \cite{hairer2006, kotyczka2019}. 
For general non-quadratic Hamiltonians, discrete gradient methods (DGM) provide a natural framework for structure-preser\-ving time discretizations, since the defining relation of a discrete gradient directly translates the continuous energy identity into a corresponding discrete one \cite{budd1999,gonzalez1996,kinon2026,mclachlan1999}. Standard discrete gradient methods, however, generally attain at most second-order convergence; a prominent example is the averaged vector field (AVF) method \cite{quispel2008}. Other second-order discrete gradients, besides the average vector field (mean value) discrete gradient \cite{harten1983, mclachlan1999}, include the Gonzalez (midpoint) discrete gradient \cite{gonzalez1996} and the symmetrized Itoh-Abe discrete gradient \cite{eidnes2022, itoh1988}. Higher-order convergence can be achieved by Petrov--Galerkin methods \cite{egger2021, giesselmann2024}.  In these methods the energy identity is generally retained only in an integral form and does not necessarily imply unconditional energy dissipativity.  Discrete gradient methods and average vector field collocation methods may be viewed as special instances of the Galerkin approach in \cite{egger2021}; see also the Galerkin approaches in \cite{li2016, wang2018}.
An alternative route towards higher-order (modified) discrete gradient methods was developed in \cite{eidnes2022} for Poisson systems, corresponding to the conservative case $R=0$. The basic idea is to perturb the structure matrix according to $\widetilde{J}=J+\mathcal{O}(h^2)$ with step size $h$, such that a conventional discrete gradient method applied to the perturbed system yields a higher-order approximation of the original dynamics. This perturbation overcomes the second-order limitation, while retaining the energy-conserving property. Extending this principle to dissipative Hamiltonian systems is not straightforward, since the higher-order perturbation generally destroys the positive semi-definiteness of the dissipation matrix. Consequently, a direct perturbation of the dissipative structure does not provide an unconditional energy-consistency guarantee.

In this work, we develop fourth-order energy-consistent average vector field discrete gradient methods (AVF-DGM) for dissipative Hamiltonian systems. In particular, we extend the perturbation principle of \cite{eidnes2022} to the dissipative setting and combine it with a congruence-based completion of the perturbed dissipation matrix. The completion is the key ingredient for incorporating the higher-order correction while preserving symmetry and positive semi-definiteness, thereby guaranteeing energy consistency.  More generally, the introduced concept of congruence-based completion provides a systematic framework for the construction of higher-order energy-consistent schemes and, in particular, allows the derivation to be extended directly to arbitrary orders $p>4$. An important special case is the purely dissipative setting $J=0$. Such systems arise naturally as subproblems in energy-consistent splitting and decomposition approaches for port-Hamiltonian systems, including fourth- and higher-order commutator-based splittings for linear \cite{moench2025} and certain nonlinear systems \cite{moench2026} as well as energy-associated decompositions \cite{bartel2025, moench2026a}. In the absence of the skew-symmetric structure matrix $J$, the dissipation correction admits alternative formulations. We therefore introduce an exponential-based modification of the dissipation matrix and investigate structure-preserving rational approximations of the resulting matrix exponential. In particular, we establish conditions under which Pad\'{e} approximants  preserve non-negativity on the relevant spectral interval and study their performance as numerical integrators for dissipative subflows in energy-consistent splitting schemes. The numerical experiments confirm the accuracy and unconditional dissipativity of the proposed methods and compare their computational efficiency with an established Petrov--Galerkin scheme.

The remainder of the paper is organized as follows. Section~\ref{sec:DGM} develops the fourth-order modification of the AVF-DGM and the corresponding structure-preserving dissipative completions, including the exponential and rational formulations for purely dissipative systems.  The numerical experiments for a damped physical pendulum and a dissipative Fermi--Pasta--Ulam system demonstrate energy consistency, fourth-order convergence, and computational efficiency in Section~\ref{sec:NumericalResults}. Finally, Section~\ref{sec:Conclusion} summarizes the main results and discusses possible extensions.

\section{Fourth-Order Unconditionally Dissipative Discrete Gradient Methods} \label{sec:DGM}

In this section, we develop fourth-order energy-consistent modified AVF-DGM for dissipative Hamiltonian systems. We first derive the order-improving perturbation and then construct a positive semi-definite congruence-based completion of the dissipative part. For the purely dissipative case $J=0$, the construction simplifies and admits an alternative formulation based on a matrix exponential. We establish structure-preserving rational approximations of the exponential, in particular Pad\'{e} approximants with even numerator degree, providing an efficient implementation.

\subsection{Standard AVF-DGM}
Consider a dissipative Hamiltonian system \eqref{eq:system} with $\mathcal{H}\in \mathcal{C}^2(\mathcal{X},\mathbb{R})$, $\mathcal{X}\subset \mathbb{R}^d$, and continuous solution $x(t)$, $t\in[0,T]$, $T>0$.
A discrete gradient method $\Phi$ applied to \eqref{eq:system} advances a given value $x_n \approx x(t_n)$ to a numerical approximation $x_{n+1} = \Phi_h(x_n) \approx x(t_n + h)$ with step size $h>0$ according to, $x_0=x(t_0)$, $t_0=0$, 
\begin{equation}\label{eq:DGM}
    \frac{x_{n+1} - x_n}{h} = [J-R]\,  \overline{\nabla} \mathcal{H}(x_n,x_{n+1}),
\end{equation}
where $\overline{\nabla} \mathcal{H} \colon \mathcal{X} \times \mathcal{X} \to \mathbb{R}^d$ denotes the discrete gradient satisfying, for all $x,y \in \mathcal{X}$,
\begin{itemize}
    \item[(i)] $\overline{\nabla} \mathcal{H}(x,x) = \nabla \mathcal{H}(x)$ \hfill (consistency),
    \item[(ii)] $\overline{\nabla} \mathcal{H}(x,y)^\top (y-x) = \mathcal{H}(y) - \mathcal{H}(x)$ \hfill (discrete chain rule).
\end{itemize}

\emph{\textbf{Energy consistency}} follows directly from these defining properties. In fact, the discrete dissipation inequality holds, 
\begin{align*}
        \mathcal{H}(x_{n+1}) - \mathcal{H}(x_n) 
        &= -h \overline{\nabla} \mathcal{H}(x_n,x_{n+1})^\top \, R \, \overline{\nabla} \mathcal{H}(x_n,x_{n+1}) \le 0, \quad \text{ for all } h>0,\\
        & \quad \text{with equality if } R=0. 
\end{align*}
If $x^*$ is an equilibrium of \eqref{eq:system}, $\mathcal{H} - \mathcal{H}(x^*)$ is a Lyapunov function, and  $x^*$ is preserved by the scheme $\Phi$, then energy consistency implies unconditional dissipativity with respect to the Hamiltonian and, consequently, preservation of the Lyapunov stability of $x^*$.

\emph{\textbf{2nd-order convergence}} is the highest order that can generally be attained by DGM for nonlinear  
$\mathcal{H}$, owing to the two-point structure of the discrete gradient, i.e., $\|x_n-x(t_n)\|=\mathcal{O}(h^2)$. In this work, we focus on the average vector field discrete gradient \cite{harten1983, mclachlan1999} 
 \begin{equation}\label{eq:AVF_DG}
        \overline{\nabla}_{\mathrm{AVF}}\mathcal{H}(x,y) = \int_0^1\nabla\mathcal{H}((1-\xi)x+\xi y) \, \mathrm{d}\xi,
\end{equation}
which gives rise to the second-order AVF-DGM. For sufficiently smooth $\mathcal{H}$, its local truncation error (consistency error) is, by Taylor expansion, given by
\begin{align}\label{eq:DGM_truncation-error}
 x(t_0 + h) - x_1 
 &= - \frac{h^3}{12} \, [J-R]\nabla^2\mathcal H(x_0)[J-R]\nabla^2\mathcal H(x_0)[J-R]\,\nabla\mathcal H(x_0) + \mathcal{O}(h^4)\\
 &= - \frac{h^3}{12} \,[J-R]\nabla^2\mathcal H(\bar x) [J-R]\nabla^2\mathcal H(\bar x)[J-R]\, \nabla\mathcal H(\bar x) + \mathcal{O}(h^5), \quad \bar x=\tfrac{1}{2}(x_0+x_1).  \nonumber
    \end{align}
with pointwise symmetric Hessian $\nabla^2 \mathcal{H}$. Expressing the leading local error in terms of the midpoint of the numerical step eliminates the fourth-order contribution from the expansion. This midpoint representation is crucial for the derivation of higher-order schemes. As the AVF-DGM is a one-step method, its convergence order follows from its consistency order under the stated regularity and standard local solvability assumptions \cite{hairer2006}.

\subsection{Order increase via perturbation principle}\label{sec:modified-DGM-general}

Aiming at a higher-order scheme, we follow the perturbation principle from \cite{eidnes2022} and perturb the system matrix $A=J-R$ of \eqref{eq:system} rather than changing the discrete gradient framework. We consider the modified AVF-DGM
\begin{align}\label{eq:modified-AVF}
\frac{x_{n+1}-x_n}{h}
&=\overline A(x_n,x_{n+1},h)\,\overline{\nabla}_{\mathrm{AVF}}\mathcal{H}(x_n,x_{n+1}), \\
&\text{ with } \lim_{\substack{y\rightarrow x\\ h\rightarrow0}}\overline A(x,y,h)=A, \nonumber
\end{align}
and choose a modification $\overline A$ such that the third-order term in the consistency error (leading-order local truncation error) of the classical AVF-DGM, cf.\ \eqref{eq:DGM_truncation-error}, is cancelled. This leads to
\begin{equation}\label{eq:Atilde}
\overline A = \widetilde A, \qquad \widetilde A(x,y,h)=A-\frac{h^2}{12}A \,\nabla^2 \mathcal{H}(\tfrac{x+y}{2}) \,A\,\nabla^2 \mathcal{H}(\tfrac{x+y}{2})\, A. 
\end{equation}
As the AVF discrete gradient is second-order accurate with respect to the midpoint, $\overline{\nabla}_{\mathrm{AVF}}\mathcal{H}(x,y)=\nabla \mathcal{H}(\tfrac{x+y}{2})+\mathcal{O}(\|y-x\|^2)$, the resulting scheme \eqref{eq:modified-AVF}-\eqref{eq:Atilde} yields a fourth-order approximation to the solution of the dissipative Hamiltonian system~\eqref{eq:system}. Moreover, the midpoint evaluation preserves the symmetry of the underlying AVF-DGM \eqref{eq:DGM}-\eqref{eq:AVF_DG}.

\begin{remark}
The perturbation principle was introduced in \cite{eidnes2022} for modified discrete gradient methods applied to energy-conservative systems, i.e., with $R=0$. There, a B-series analysis is used to derive order conditions in a more general setting, including general second-order discrete gradients and state-dependent structure matrices. In the present work, we derive the required correction directly from a Taylor expansion of the local truncation error for dissipative Hamiltonian systems. For constant $J$ and $R$, the resulting order condition takes the same form with $A=J-R$.
\end{remark}

\subsection{Energy consistency through congruence-based completion}\label{sec:completion}
To ensure energy con\-sis\-tency of the modified AVF-DGM \eqref{eq:modified-AVF}, we require the modification to retain the decomposition into a pointwise skew-symmetric structure matrix and a pointwise symmetric positive semi-definite dissipation matrix function, 
$$\overline A=\overline J - \overline R, \qquad \overline J(x,y,h)=-\overline J(x,y,h)^\top, \quad \overline R(x,y,h)=\overline R(x,y,h)^\top \succeq 0.$$

For the order-improving correction $\widetilde{A}$ in \eqref{eq:Atilde}, substituting $A=J-R$ and separating the skew-symmetric and symmetric contributions yields
\begin{align}\label{eq:Atilde-decomposition} 
\widetilde A =\widetilde J-\widetilde R, \qquad \text{ with } \, \widetilde J(x,y,h) &= J-\frac{h^2}{12}(T_0(\tfrac{x+y}{2})+T_2(\tfrac{x+y}{2})), \\ 
\widetilde R(x,y,h) &= R-\frac{h^2}{12}(T_1(\tfrac{x+y}{2})+T_3(\tfrac{x+y}{2})),\nonumber
\end{align}
where the higher-order terms $T_i$ are indexed according to the number of occurrences of $R$, 
\begin{align*}
T_0(x) &=J\nabla^2\mathcal{H}(x)J\nabla^2\mathcal{H}(x)J,
\\
T_1(x)&= R\nabla^2\mathcal{H}(x)J\nabla^2\mathcal{H}(x)J +J\nabla^2\mathcal{H}(x)R \nabla^2\mathcal{H}(x)J +J\nabla^2\mathcal{H}(x)J \nabla^2\mathcal{H}(x)R,
\\
T_2(x)&=R\nabla^2\mathcal{H}(x)R\nabla^2\mathcal{H}(x)J+R\nabla^2\mathcal{H}(x)J\nabla^2\mathcal{H}(x)R+J\nabla^2\mathcal{H}(x)R\nabla^2\mathcal{H}(x)R,
\\
T_3(x) &=R\nabla^2\mathcal{H}(x)R\nabla^2\mathcal{H}(x)R.
\end{align*}
By construction, $\widetilde J$ is pointwise skew-symmetric and consistent with $J$, i.e., $\lim_{y\rightarrow x,h\rightarrow0}\widetilde J(x,y,h)=J$. Likewise, $\widetilde R$ is pointwise symmetric and consistent with $R$. In contrast to $R$, however, $\widetilde R(x,y,h)$ is not necessarily positive semi-definite. This loss of structure is the central difficulty in extending the higher-order perturbation approach to dissipative Hamiltonian systems, since the resulting scheme is no longer unconditionally dissipative.  We restore this property by means of a suitable completion.

\begin{definition}[Congruence-based completion] \label{def:congruence-completion}
Consider a dissipative Hamiltonian system \eqref{eq:system} and the modified AVF-DGM \eqref{eq:modified-AVF} with \eqref{eq:Atilde-decomposition}.
A matrix-valued function $\overline R$ is called a \emph{fourth-order congruence-based completion} of $\widetilde R$ if $\overline R=\widetilde R+O(h^4)$ and
\begin{equation}\label{eq:congruence-completion}
\overline R(x,y,h)=\sum_{j=1}^m \alpha_j \ \Psi_j(x,y,h)^\top R \ \Psi_j(x,y,h), \qquad \alpha_j\geq0,
\end{equation}
for suitable matrix functions $\Psi_j$.
\end{definition}

\begin{lemma}\label{lem:congruence-structure}
Every congruence-based completion \eqref{eq:congruence-completion} is pointwise symmetric positive semi-definite.
\end{lemma}

\begin{proof}
The symmetry of $\overline{R}(x,y,h)$ follows directly from the symmetry of $R$.
Moreover, for any $v\in\mathbb R^d$,
$v^\top\overline Rv = \sum_{j=1}^m \alpha_j (\Psi_jv)^\top R(\Psi_jv) \geq0,$
because $R\succeq0$ and $\alpha_j\geq0$.
Thus, $\overline R(x,y,h)=\overline R(x,y,h)^\top\succeq0$.
\end{proof}
The representation \eqref{eq:congruence-completion} may also be interpreted in terms of Gram matrices; see, e.g., \cite{golub2013}. Using a factorization $R=L^\top L$, we obtain
$\overline R = \sum_{j=1}^m \alpha_j (L\Psi_j)^\top(L\Psi_j)$,
so that $\overline R$ is a non-negative linear combination of Gram matrices. The matrix functions  $\Psi_j$ are subsequently chosen to satisfy the required order conditions.

\begin{theorem} \label{thm:completion}
Consider a dissipative Hamiltonian system \eqref{eq:system} and the modified AVF-DGM \eqref{eq:modified-AVF}.
Let the modification $\overline{A}=\overline{J}-\overline{R}$ be composed of $\overline{J}=\widetilde{J}$ \eqref{eq:Atilde-decomposition} and $\overline R=\sum_{j=1}^m \alpha_j \Psi_j^\top R\Psi_j$,  $\alpha_j\geq0$, from \eqref{eq:congruence-completion}. Suppose that  $\Psi_j(x,y,h)=I+hC_j(x,y)+h^2D_j(x,y)$ and that
\begin{align*}\label{eq:completion-condition-1}
\sum_{j=1}^m\alpha_j &=1, 
& \sum_{j=1}^m\alpha_j (C_j^\top R+RC_j) =0,\\
\sum_{j=1}^m\alpha_j\left(C_j^\top RC_j+D_j^\top R+RD_j\right)\Big|_{(x,y)}&=-\tfrac{1}{12}(T_1+T_3)\Big|_{(x+y)/2},
\quad &\sum_{j=1}^m\alpha_j(C_j^\top RD_j+D_j^\top RC_j)=0.
\end{align*}
Then, the resulting modified AVF-DGM is energy-consistent and of fourth order. 
\end{theorem}

\begin{proof}
Expanding \eqref{eq:congruence-completion} gives
\begin{align*}
\overline R(x,y,h)=\sum_{j=1}^m\alpha_j\Bigl[&R +h(C_j^\top R+RC_j)+h^2(C_j^\top RC_j+D_j^\top R+RD_j)+h^3(C_j^\top RD_j+D_j^\top RC_j)\\
&+h^4D_j^\top RD_j\Bigr] \Big|_{(x,y)}.
\end{align*}
The assumed relations therefore imply
\begin{equation}
\overline R(x,y,h)=R-\frac{h^2}{12}(T_1+T_3)\Big|_{(x+y)/2}+h^4\sum_{j=1}^m\alpha_jD_j^\top RD_j\Big|_{(x,y)}.
\end{equation}
Hence, $\overline R=\widetilde R+O(h^4)$ and is consistent with $R$, as $\lim_{y\rightarrow x,h\rightarrow0}\overline R(x,y,h)=R$. By Lemma~\ref{lem:congruence-structure}, $\overline R$ is pointwise symmetric positive semi-definite. Together with the pointwise skew-symmetry and consistency of $\overline J=\widetilde{J}$ with $J$, the resulting modified AVF-DGM is thus energy-consistent and fourth-order.
\end{proof}

\begin{example}[Fourth-order energy-consistent scheme]\label{ex:c-b_method}
    Consider $m=2$ with $\alpha_1 =\alpha_2= \tfrac{1}{2}$. Choose $C_1= C$, $C_2 = -C$ and $D_1=D_2=D$, where 
    \begin{align*}
        C(x,y)&=\tfrac{1}{2\sqrt{3}} \nabla^2 \mathcal{H}(\tfrac{x+y}{2}) \, J,\\
         D(x,y)&=-\tfrac{1}{12}\nabla^2 \mathcal{H}(\tfrac{x+y}{2}) \, J \, \nabla^2 \mathcal{H}(\tfrac{x+y}{2}) \, J -\tfrac{1}{24}\nabla^2 \mathcal{H}(\tfrac{x+y}{2}) \, R \, \nabla^2 \mathcal{H}(\tfrac{x+y}{2}) \, R.
    \end{align*}
Then,  $\overline{R} = \tfrac{1}{2}\Psi_1^\top R\Psi_1 + \tfrac{1}{2} \Psi_2^\top R\Psi_2$ with $\Psi_{1,2} = I \pm hC +h^2 D$, is a fourth-order congruence-based completion. The resulting energy-consistent AVF-DGM is
    \begin{align}\label{eq:DGM4}
        \frac{x_{n+1}-x_n}{h}&=\overline A(x_n,x_{n+1},h)\overline{\nabla}_{\mathrm{AVF}}\mathcal{H}(x_n,x_{n+1}),\\
        \overline A(x,y,h) &= A-\frac{h^2}{12}A\nabla^2 \mathcal{H}(\tfrac{x+y}{2}) A\nabla^2 \mathcal{H}(\tfrac{x+y}{2})A - h^4 D(x,y)^\top RD(x,y), \quad A=J-R. \nonumber
    \end{align}
  Compared to $\widetilde{A}$ \eqref{eq:Atilde}, the additional $(h^4D^\top RD)$  term in \eqref{eq:DGM4} does not affect fourth-order consistency and is introduced solely to complete the dissipative matrix in a positive semi-definite manner.
 \end{example}

\subsection{Purely dissipative systems}\label{sec:purely-dissipative}

In the special case $J=0$, the dissipative system \eqref{eq:system} reduces to
\begin{equation}\label{eq:purely-dissipative}
\dot{x}=-R\nabla \mathcal{H}(x), \quad x_0(0)=x_0, \qquad R=R^\top \succeq 0.
\end{equation}
Moreover, since $T_0=T_1=T_2=0$, the fourth-order correction $\widetilde{A}$ \eqref{eq:Atilde-decomposition} simplifies to
\begin{equation} \label{eq:Rtilde-pure}
\widetilde A=-\widetilde R, \qquad
\widetilde R(x,y,h) = R-\frac{h^2}{12} R\nabla^2\mathcal{H}(\tfrac{x+y}{2}) R\nabla^2\mathcal{H}(\tfrac{x+y}{2})R
\end{equation}
and allows for alternative completions, such as the construction of an \emph{exponential completion}.

\begin{theorem}
Consider the purely dissipative system \eqref{eq:purely-dissipative}.
The matrix function
\begin{equation} \label{eq:Rbar-exp}
\overline R_{\exp}(x,y,h)
=\exp (Z(\tfrac{x+y}{2},h))R, \qquad Z(x,h)=-\frac{h^2}{12}R\nabla^2\mathcal{H}(x)R\nabla^2\mathcal{H}(x)
\end{equation}
is a fourth-order symmetric positive semi-definite completion of $\widetilde{R}$ \eqref{eq:Rtilde-pure}, and the modified AVF-DGM \eqref{eq:modified-AVF} with $\overline{A}=-\overline{R}_{\exp}$ is energy-consistent and of fourth order.
\end{theorem}

\begin{proof}
The correction $\widetilde R$ is the second-order truncation (Taylor approximation) of the matrix exponential, hence
\begin{align*}
\overline R_{\exp}(x,y,h)
&=\exp\left(-\frac{h^2}{12}R\nabla^2\mathcal{H}R\nabla^2\mathcal{H}\right)\bigg|_{(x+y)/2}R\\
&=\left(I-\frac{h^2}{12}R\nabla^2\mathcal{H}R\nabla^2\mathcal{H}+ \mathcal{O}(h^4)\right)\bigg|_{(x+y)/2} R
=\widetilde{R}(x,y,h)+\mathcal{O}(h^4),
\end{align*}
implying the consistency of $\overline R_{\exp}$ with $R$. To establish the structural properties, we introduce
\begin{equation}\label{eq:Z}
\widehat Z(x,h)=-\frac{h^2}{12}R^{1/2}\nabla^2\mathcal{H}(x)R\nabla^2\mathcal{H}(x)R^{1/2} 
\end{equation}
with square root $R^{1/2}$, i.e., $R^{1/2} R^{1/2}=R$ and $R^{1/2}=(R^{1/2})^\top \succeq 0$. Since $R^{1/2} \widehat Z^k R^{1/2}=Z^kR$, $k\in \mathbb{N}_0$, the power series of the exponential yields
\begin{equation*}
\overline R_{\exp}(x,y,h) =R^{1/2}\exp(\widehat Z|_{((x+y)/2,h)})R^{1/2}.
\end{equation*}
As $\widehat Z$ is pointwise symmetric and negative semi-definite, $\exp(\widehat Z)$ is pointwise symmetric positive semi-definite according to the spectral mapping theorem \cite{higham2008}. Consequently,
$\overline R_{\exp}(x,y,h)=\overline R_{\exp}(x,y,h)^{\top}\succeq0$. By construction, the exponential completion yields the desired accuracy and energy consistency of the modified AVF-DGM. 
\end{proof}

The exponential completion is a special case of the congruence-based construction. Setting
\begin{equation*}
\Psi_{\exp}(x,y,h)=\exp\left(\tfrac{1}{2} Z(\tfrac{x+y}{2},h)^\top \right),
\end{equation*}
we obtain
$\overline R_{\exp}= \Psi_{\exp}^\top R\Psi_{\exp}.$
Thus, the exponential completion corresponds to Definition~\ref{def:congruence-completion} with $m=1$ and $\alpha_1=1$. Moreover, it is closely related to the two-term congruence-based completion from Example~\ref{ex:c-b_method}. In the purely dissipative case, $C$ and $D$ reduce to 
\begin{equation*}
        C(x,y)=0, \qquad D(x,y) = -\tfrac{1}{24}\nabla^2\mathcal{H}(\tfrac{x+y}{2})R\nabla^2\mathcal{H}(\tfrac{x+y}{2}) R.
\end{equation*}
The resulting matrix functions $\Psi_{1,2} = I+\tfrac{1}{2}Z^\top$ are second-order truncations of the matrix exponential $\Psi_{\exp}$, i.e., $\Psi_{\exp} = \Psi_{1,2} + \mathcal{O}(h^4)$ as $Z=\mathcal{O}(h^2)$.

\begin{remark}
For general dissipative systems \eqref{eq:system} with $J \neq 0$, the exponential completion requires the dissipation matrix to be invertible, i.e., $R=R^\top \succ 0$, which might be quite restrictive in practice. With $\widetilde R(x,y,h) = R + h^2 T(\tfrac{x+y}{2})$, cf.\ \eqref{eq:modified-AVF}, the exponential completion is then
  \begin{equation*}
        \overline R_{\exp}(x,y,h) = \exp(h^2T(\tfrac{x+y}{2})R^{-1}) \,R, \qquad T(x)= -\tfrac{1}{12}(T_1(x)+T_3(x)).
  \end{equation*}
The connection to the congruence-based completion is given by $\Psi_{\exp}= \exp(\tfrac{h^2}{2}R^{-1}T)$. 
\end{remark}

In practice, the matrix exponential motivates the use of rational functions. We thus investigate rational approximations, in particular Pad\'{e} approximations \cite{baker1996}, with the aim of retaining the structural properties of the exponential completion. Replacing $\exp(Z)$ in \eqref{eq:Rbar-exp} by a rational function $r(Z)$ leads to a \emph{rational completion}, and, in the case of a Pad\'{e} approximant, to a \emph{Pad\'{e}-based completion}.

\begin{theorem}\label{cor:rat_completion}
Let $r \colon \mathbb{R}\to\mathbb{R}$ be a rational function with no poles on $(-\infty,0]$. Suppose that $r(z) = \exp(z)+\mathcal{O}(z^2)$ as $z\to0$ and  $r(z) \geq0$ for all  $z\leq0$.
Then, with $ Z$ given by \eqref{eq:Rbar-exp},
\begin{equation*}
    \overline R_{\mathrm{rat}}(x,y,h)=r(Z(\tfrac{x+y}{2},h))\,R
\end{equation*}
 is a fourth-order symmetric positive semi-definite completion of $\widetilde{R}$ \eqref{eq:Rtilde-pure}.
 \end{theorem}

\begin{proof}
Consider $\widehat{Z}$ from \eqref{eq:Z}. Since $Z(x,h)=R^{1/2} K(x,h)$ and $\hat{Z}(x,h)=K(x,h) R^{1/2}$ with $K(x,h)=-\tfrac{h^2}{12}R^{1/2}\nabla^2 \mathcal{H}(x)R \nabla^2 \mathcal{H}(x)$, the matrices have the same non-zero eigenvalues. Moreover $\widehat{Z}(x,y,h) \preceq  0$, so all eigenvalues of $Z$ are real and non-positive. Since $r$ has no poles on $(-\infty,0]$, both matrix functions $r(Z)$ and $r(\widehat{Z})$ are well-defined, i.e.\ their denominator polynomials $q(Z)$ and $q(\widehat Z)$ are pointwise invertible. The intertwining relation $ZR^{1/2}=R^{1/2}\widehat{Z}$ extends hence from polynomials to rational functions, yielding
$r(Z)R^{1/2}=R^{1/2}r(\widehat{Z})$. Thus,  $\overline R_{\mathrm{rat}}(x,y,h)=\,R^{1/2}r(\widehat{Z}(\tfrac{x+y}{2},h))\,R^{1/2}$.
Since $\widehat Z$ is pointwise symmetric, it admits an orthogonal diagonalization
$\widehat Z=Q\operatorname{diag}(\lambda_1,\ldots,\lambda_d)Q^\top$ with eigenvalues $\lambda_i\leq0$.
The assumption $r(z)\geq 0$ for $z\leq 0$ implies $r(\widehat{Z})=Q\operatorname{diag}(r(\lambda_1),\ldots,r(\lambda_d))Q^\top$ to be pointwise symmetric positive semi-definite. Hence, $\overline R_{\mathrm{rat}}(x,y,h)=\overline R_{\mathrm{rat}}(x,y,h)^\top \succeq 0$.
Finally, since $r(z)=\exp(z)+\mathcal{O}(z^2)$ as $z\rightarrow 0$ and $\widehat Z=\mathcal{O}(h^2)$ as $h\rightarrow 0$, we obtain $r(\widehat Z)=\exp(\widehat{Z})+\mathcal{O}(h^4)$, and therefore
$$\overline R_{\mathrm{rat}} = \overline{R}_{\exp} +\mathcal{O}(h^4) = \widetilde R+\mathcal{O}(h^4).$$
\end{proof}

Let $r_{k,m}$ denote the $[k/m]$-Pad\'{e} approximant of the exponential function, $k$, $m\in \mathbb{N}_0$,
\begin{align}\label{eq:pade-rational}
r_{k,m}(z) &=\frac{N_{k,m}(z)}{D_{k,m}(z)},\\
N_{k,m}(z)&=
\sum_{i=0}^k \binom{k}{i} \frac{(k+m-i)!}{(k+m)!}\,z^i,
\qquad
D_{k,m}(z)=
\sum_{i=0}^m \binom{m}{i} \frac{(k+m-i)!}{(k+m)!}\,(-z)^i.\nonumber
\end{align}
It satisfies $r_{k,m}(z)=\exp(z)+\mathcal{O}(z^{k+m+1})$ as $z \rightarrow 0$ and possesses no poles on $(-\infty, 0]$, since
$D_{k,m}(z)\geq D_{k,m}(0)=1$ for all $z\leq 0$. The following result characterizes the positivity property for the Pad\'e approximants. Their zero distribution has been studied extensively; see, e.g., \cite{saff1975}. Here, we require a more specific characterization of their real zeros.

\begin{lemma}\label{lem:pade-numerator-zeros}
For $k,m\geq0$, the numerator polynomial $N_{k,m}$ has at most one distinct real zero on $(-\infty,0)$. If $k$ is odd, this zero exists and has odd multiplicity. If $k$ is even, $N_{k,m}$ does not change sign on $\mathbb{R}$.
\end{lemma}

\begin{proof}
The numerator can be represented as a Kummer polynomial, $N_{k,m}(z)=M(-k,-k-m,z)$, cf.\ \cite{baker1996},
and thus satisfies the Kummer differential equation
\begin{equation*}
zN^{\prime \prime}_{k,m}(z) -(k+m+z)N_{k,m}^\prime(z)+kN_{k,m}(z)=0.
\end{equation*}
For $z<0$, division by $z$ and multiplication by the integrating factor
$p(z)=|z|^{-(k+m)}\exp(-z)>0$
transform this equation into self-adjoint form,
\begin{equation}\label{eq:pade-self-adjoint}
\bigl(p(z)N_{k,m}^\prime(z)\bigr)^\prime + q(z)N_{k,m}(z) =0,
\end{equation}
where $q(z)=\tfrac{k}{z}p(z)=-k|z|^{-(k+m)-1}\exp(-z) \leq 0$.
Suppose, for contradiction, that $N_{k,m}$ has two distinct zeros $z_1<z_2<0$. Choose two consecutive distinct zeros, so that $N_{k,m}$ does not vanish on $(z_1,z_2)$. Multiplying \eqref{eq:pade-self-adjoint} by $N_{k,m}$ and integrating over $(z_1,z_2)$ gives
\begin{align*}
0 = \int_{z_1}^{z_2} N_{k,m}\bigl(p N_{k,m}'\bigr)'\,\mathrm{d}z + \int_{z_1}^{z_2}qN_{k,m}^2\,\mathrm{d}z \quad 
\Longrightarrow \quad \int_{z_1}^{z_2}p\bigl(N_{k,m}'\bigr)^2\,\mathrm{d}z =  \int_{z_1}^{z_2}q N_{k,m}^2\,\mathrm{d}z\leq 0
\end{align*}
with integration by parts and $N_{k,m}(z_1)=N_{k,m}(z_2)=0$. Since $p>0$, the left-hand side is non-negative. Hence it must vanish. In particular, $N_{k,m}'\equiv 0$ on $(z_1,z_2)$, which, together with $N_{k,m}(z_1)=0$, implies $N_{k,m}\equiv 0$ on $(z_1,z_2)$, contracting the fact that $N_{k,m}$ is a nontrivial polynomial. Hence $N_{k,m}$ has at most one distinct real zero on $(-\infty,0)$. This argument is a special case of the Sturm disconjugacy theorem for self-adjoint equations; see, e.g., \cite{barutello2021}.

Since $N_{k,m}(0)=1$ and all coefficients of $N_{k,m}$ are non-negative, we have $N_{k,m}(z)>0$ for $z\geq0$, and hence all real zeros are exclusively contained in $(-\infty,0)$. For $k=0$,  $N_{0,m}\equiv1$, so there are no real zeros.
Let $k\geq1$. Since $N_{k,m}$ is a real polynomial of degree $k$, its non-real zeros occur in conjugate pairs. Consequently, the total multiplicity of its real zeros has the same parity as $k$. As there is at most one distinct real zero, the following holds: if $k$ is odd, the polynomial must have a real zero $z_0<0$ with odd multiplicity, and $N_{k,m}$ changes sign at $z_0$.
If $k$ is even, any real zero must have even multiplicity. Thus, $N_{k,m}$ does not change sign on $\mathbb{R}$.
\end{proof}

\begin{theorem}\label{thm:pade-positivity}
Let $r_{k,m}$ be the $[k/m]$-Pad\'{e} approximant of the exponential function \eqref{eq:pade-rational}. Then
\begin{equation}\label{eq:pade-positivity} 
r_{k,m}(z)\geq0
\qquad\text{for all }z\leq0
\end{equation}
if and only if $k$ is even.
\end{theorem}

\begin{proof}
For $z\leq0$, $D_{k,m}(z)\geq1$, and hence the sign of $r_{k,m}$ is determined by the numerator $N_{k,m}$. In particular, $N_{k,m}(0)=1$ holds for $k,m\in \mathbb{N}_0$.
Let $k$ be even. By Lemma~\ref{lem:pade-numerator-zeros}, $N_{k,m}$ does not change sign on $\mathbb{R}$. Hence, $N_{k,m}(z)\geq0$ for $z\leq0$, yielding \eqref{eq:pade-positivity}.
Conversely, let $k$ be odd, then $N_{k,m}$ has a unique real zero $z_0<0$ of odd multiplicity. Thus, $N_{k,m}(z)<0$ for sufficiently negative $z$, and consequently $r_{k,m}(z)<0$ for some $z<0$. Hence \eqref{eq:pade-positivity} cannot hold.
\end{proof}

Theorem~\ref{thm:pade-positivity} shows that unconditional positive semi-definiteness of the rational completion restricts the numerator degree of the Pad\'{e} approximant to even values. 

\begin{example}
The $[0/1]$-Pad\'{e} approximant  $r_{0,1}$ meets the requirements of Theorem~\ref{cor:rat_completion}, yielding a positive semi-definite completion for arbitrary step sizes $h>0$, i.e.
\begin{equation}\label{eq:Pade01}
r_{0,1}(z) = \left(1-z\right)^{-1} \quad \Rightarrow \quad  \overline{R}_{0,1}(x,y,h) = \left( I- Z\right)^{-1}\big|_{((x+y)/2,h)} \,R.
\end{equation}
The $[0/2]$- and $[2/0]$-Pad\'{e} approximants  are third-order approximations of the exponential function that are positive on the entire real axis,
\begin{equation*}
r_{0,2}(z) = \left(1-z+\tfrac{z^2}{2}\right)^{-1}, \qquad  r_{2,0}(z) = 1+z+\tfrac{z^2}{2},
\end{equation*}
and give rise to the respective completions
\begin{align*}
\overline{R}_{0,2}(x,y,h) &= \left( I- Z + \tfrac{1}{2} Z^2 \right)^{-1}\big|_{((x+y)/2,h)} \,R, \\ 
\overline{R}_{2,0}(x,y,h) &= \left( I+Z + \tfrac{1}{2} Z^2 \right)\big|_{((x+y)/2,h)} \,R. 
\end{align*}
With $\overline{A}=-\overline{R}_{k,m}$, all resulting AVF discrete gradient schemes are fourth-order and energy-consistent for the purely dissipative system \eqref{eq:purely-dissipative}.
The polynomial completion $\overline{R}_{2,0}$  does not offer a decisive computational advantage, since all variants require the solution of a nonlinear system arising from the discrete gradient at each times step. The nonlinear systems differ due to the respective completions; however, the computational costs for the systems with $\overline{R}_{2,0}$  and $\overline{R}_{0,2}$ are comparable. In contrast, the use of  $\overline{R}_{0,1}$ avoids the evaluation of the $Z^2$ term.
\end{example}

\begin{remark}
The diagonal $[1/1]$-Pad\'{e} approximation, also known as the Cayley transform \cite{hairer2006},
\begin{equation*}
r_{1,1}(z)=\left(1-\tfrac{z}{2}\right)^{-1}\,\left(1+\tfrac{z}{2}\right)
\end{equation*}
satisfies the required accuracy condition, but is non-negative only for $z\in[-2,0]$.
Its use requires a step-size restriction ensuring that the spectrum of $Z(x,h)$ remains in $[-2,0]$, and hence does not yield unconditional dissipativity. The same issue arises for the diagonal Pad\'{e} approximants $r_{kk}$ with odd $k$, which are commonly used in numerical software, see, e.g., \cite{almohy2010}.
\end{remark}

\section{Numerical Results}\label{sec:NumericalResults}

In this section, we investigate the two central properties of the proposed methods: energy consistency and fourth-order accuracy. First, we demonstrate the preservation of the dissipative structure for a damped physical pendulum, with particular emphasis on the role of the positive semi-definite completion. Second, we investigate the convergence behavior and computational efficiency of the methods for a dissipative Fermi--Pasta--Ulam (FPU) system.

The implementation is carried out in Python 3.14.0.
For all numerical experiments, the reference solution is computed using {\texttt{scipy.integrate.solve\_ivp}}, the adaptive eighth-order Dormand-Prince method DOP853 with relative and absolute tolerances $10^{-12}$ and $10^{-14}$, respectively. The nonlinear systems arising from the implicit schemes are solved using {\texttt{scipy.optimize.root}}, inexact Newton’s method with Krylov approximations for the inverse Jacobian and a relative error tolerance of $10^{-12}$ between consecutive iterates. All integrals are evaluated to machine precision using Gauss-Legendre quadrature with a degree of exactness sufficient for the respective integrands.

\subsection{Preservation of the dissipative structure} \label{sec:NumericalResults_DissipativeStructure}

As a benchmark, we consider the damped physical pendulum
\begin{equation*}
\dot{x} = \begin{pmatrix}\dot{p} \\ \dot{q}\end{pmatrix} =
\left[ \begin{pmatrix} 0 & -1 \\ 1 & 0 \end{pmatrix} -
\begin{pmatrix} r & 0 \\ 0 & 0 \end{pmatrix} \right]
\begin{pmatrix} p \\ \sin(q) \end{pmatrix},
\end{equation*}
with Hamiltonian $\mathcal{H}(p,q) = \tfrac{1}{2}p^2 + 1-\cos(q)$, damping coefficient $r>0$, time interval $[0,T]$ and initial value $x(0)=(0.5,\pi/2)^\top$ unless stated otherwise. 
The simplicity of this system allows us to characterize analytically the condition under which the uncompleted fourth-order correction $\widetilde R$ loses positive semi-definiteness and to derive a state-independent critical step size. This provides a convenient setting to investigate the role of the congruence-based completion and its impact on the dissipative structure.

\begin{figure}[tb]
    \centering
    \includegraphics{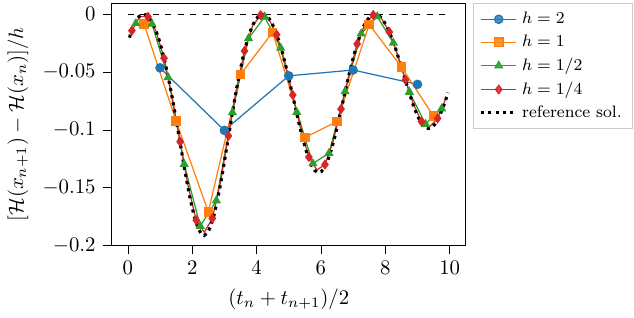}
    \caption{Damped physical pendulum, $r=0.1$. Discrete energy derivative of \texttt{DGM4} for different step sizes $h$, compared with reference solution.}
    \label{fig:pendulum_dissipation-test}
\end{figure}

\begin{figure}[b]
    \centering
    \includegraphics{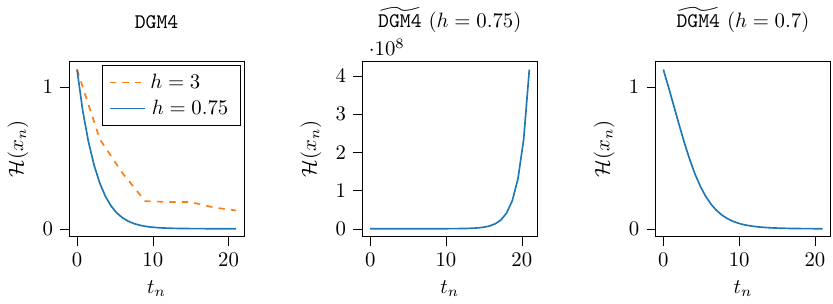}
    \caption{Damped physical pendulum, $r=5$. Energy behavior over time for \texttt{DGM4} with several $h$ (left) and for $\widetilde{\mathtt{DGM4}}$ with $h=0.75$ (center) and $h=0.7$ (right).}
    \label{fig:instabilities}
\end{figure}

As a representative example, we compare here the fourth-order energy-consistent AVF discrete gradient scheme from \eqref{eq:DGM4} (\texttt{DGM4}) with the uncompleted variant ($\widetilde{\mathtt{DGM4}}$), obtained by using $\widetilde{A}$ \eqref{eq:Atilde-decomposition}.
We first verify the energy consistency of \texttt{DGM4}. For the damped pendulum with $r=0.1$, $T=10$ and various step sizes $h \in \{1/4,1/2,1,2\}$, Figure~\ref{fig:pendulum_dissipation-test} illustrates that the discrete dissipation inequality
\begin{equation*}
\tfrac{1}{h}(\mathcal{H}(x_{n+1})-\mathcal{H}(x_n))\leq 0
\end{equation*}
holds for every trajectory and step size. Already for $h=1/4$, the numerical energy decay closely follows the reference solution. Thus, the proposed fourth-order modification preserves the dissipative structure independently of the step size. The unconditional dissipativity implies unconditional stability of the equilibrium under the stated assumptions. In contrast, methods that do not satisfy the discrete dissipation inequality may exhibit numerical energy growth and, consequently, qualitatively incorrect dynamics.

The role of the completion becomes here apparent for stronger dissipation. 
In \texttt{DGM4} the dissipation matrix function is pointwise positive semi-definite,
\begin{equation*}
\overline{R} = r \begin{pmatrix}
\left(1 + \frac{h^2}{24} (2 \cos(q) - r^2) \right)^2 & 0 \\
0 & \frac{h^2}{12}
\end{pmatrix} \succeq 0, \qquad \text{for every } h>0,
\end{equation*}
whereas $\widetilde{R}$ of $\widetilde{\mathtt{DGM4}}$ 
\begin{equation*}
\widetilde{R}
= r\begin{pmatrix}
1 + \frac{h^2}{12} (2 \cos(q) - r^2) & 0 \\
0 & \frac{h^2}{12}
\end{pmatrix}
\end{equation*}
can become indefinite for sufficiently large step sizes. While the second eigenvalue $\lambda_2$ of $\widetilde{R}$ is always positive, i.e., $\lambda_2=r h^2/12 > 0$, the first one $\lambda_1$ may change the sign.  Let $r>\sqrt{2}$, then 
\begin{align}\label{eq:h-crit}
\lambda_1=r (1+\frac{h^2}{12}(2\cos(q)-r^2))< 0 \quad \text{ for } \quad h> \sqrt{\frac{12}{r^2-2\cos(q)}}.
\end{align}
A critical state-independent step-size threshold for stability is hence $h_{\mathtt{crit}}=\sqrt{{12}/({r^2-2})}$.
The loss of positive semi-definiteness is reflected directly in the numerical dynamics as shown in Figure~\ref{fig:instabilities} for $r=5$. The numerical solution of $\widetilde{\mathtt{DGM4}}$ is stable for $h=0.7<h_{\mathtt{crit}}\approx 0.7223$ but exhibits numerical instability for $h=0.75>h_{\mathtt{crit}}$. In contrast,  \texttt{DGM4} remains dissipative even for large step sizes, see, e.g., $h=3$.

\begin{figure}[tb]
    \centering
    \includegraphics[width=\textwidth]{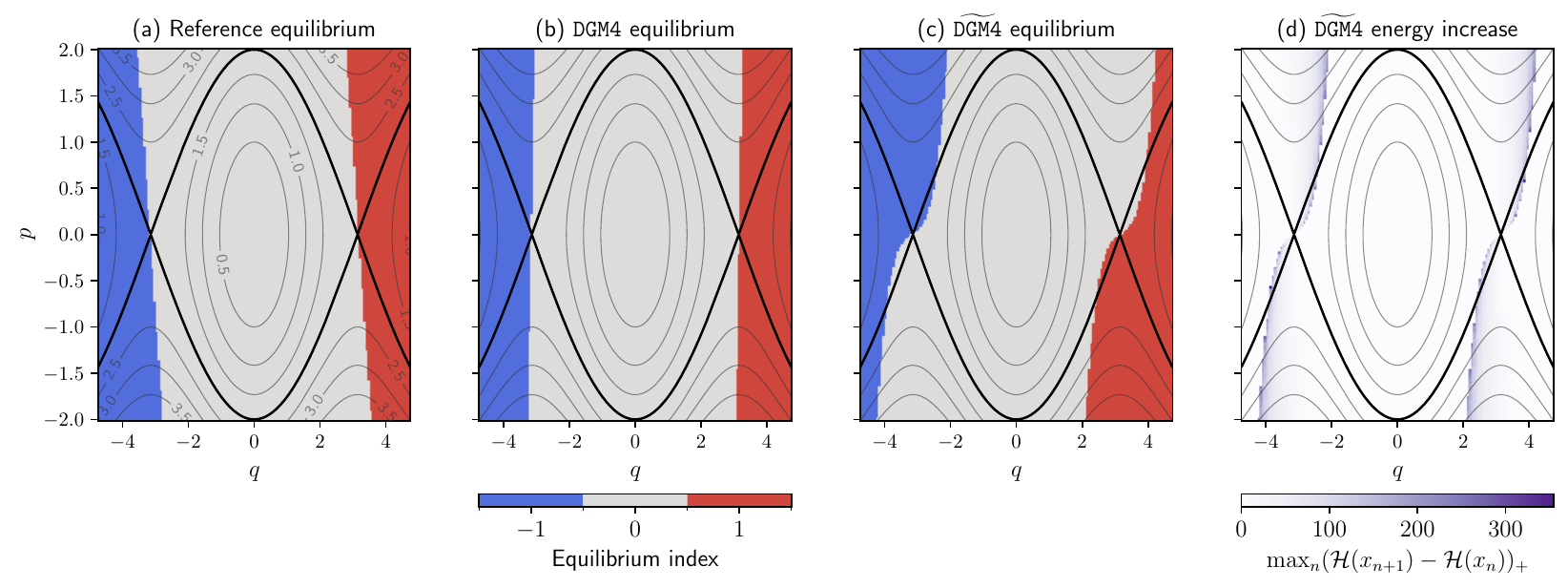}
\caption{
Damped physical pendulum, $r=5$, $t\in[0,21]$ and $h=0.7$. Grey curves indicate level sets of $\mathcal H$, and the black curve the separatrix $\mathcal H=2$.
Panels (a)-(c) show the index $j$ of the attracting equilibrium $(p,q)^\star=(0,2\pi j)$ associated with each initial value:
(a) reference solution, 
(b) \texttt{DGM4}, 
(c) $\widetilde{\mathtt{DGM4}}$. 
Panel (d) shows the largest positive one-step energy increment for $\widetilde{\mathtt{DGM4}}$.
}
\label{fig:basins}
\end{figure}

The phenomenon can also be observed globally in the state space. 
Consider an equidistant $150\times150$ grid of initial values $(p,q)(0)\in[-3\pi/2,3\pi/2]\times[-2,2]$ and classify each trajectory according to the index $j$ of the equilibrium $(p,q)^\star=(0,2\pi j)$ to which it converges. The resulting basins of attraction are compared with the reference solution in Figure~\ref{fig:basins} for $r=5$ and $h=0.7$.
For \texttt{DGM4}, the correct attracting equilibrium is obtained for $94.9\%$ of the initial values, and the dissipation inequality is satisfied throughout the experiment. In contrast, $\widetilde{\mathtt{DGM4}}$ identifies the correct equilibrium for only $79.4\%$ of the initial values, and $50.7\%$ of its trajectories exhibit at least one positive one-step energy increment due to the indefinite $\widetilde R$, when $2\cos(q)<r^2-12/h^2$, cf.\ \eqref{eq:h-crit}. The completion is therefore essential not only for guaranteeing unconditional dissipation, but also for preserving the qualitative behavior of the computed dynamics.

\subsection{Accuracy and computational efficiency}

As second benchmark, we consider a dissipative modification of the Fermi--Pasta--Ulam (FPU) system \cite{galgani1992},
\begin{align*}
\dot x = \begin{pmatrix} \dot{p} \\ \dot{q} \end{pmatrix} &=
\left[\begin{pmatrix} 0 & -I \\ I & 0 \end{pmatrix}-\begin{pmatrix} r I & 0 \\ 0 & 0 \end{pmatrix}\right]
\begin{pmatrix} \nabla_p \mathcal{H}(p,q) \\ \nabla_q \mathcal{H}(p,q) \end{pmatrix},\\
\mathcal{H}(p,q) &=\frac{p^\top p}{2}+ \frac{\omega^2}{4}\sum_{i=1}^m (q_{2i}-q_{2i-1})^2+ \sum_{i=0}^m (q_{2i+1}-q_{2i})^4, \label{eq:FPU_Hamiltonian}
\end{align*}
with state $x(t)\in\mathbb{R}^{4m}$, $t\in [0,T]$, damping coefficient $r>0$, spring coefficient $\omega >0$, and boundary values $q_0=q_{2m+1}=0$.
The system describes a chain of $2m$ mass points connected by alternating linear and nonlinear springs and fixed at both ends. Its scalable dimension makes it a suitable benchmark for evaluating the computational efficiency of the  fourth-order approaches.

We assess the accuracy and computational cost of the proposed energy-consistent discrete gradient methods and particularly compare their performance with the higher-order energy-consistent Petrov--Galerkin method of \cite{giesselmann2024} with $k=2$ (\texttt{PGM}) and the standard second-order AVF-DGM \eqref{eq:DGM}-\eqref{eq:AVF_DG} (\texttt{DGM2}). We consider here the AVF discrete gradient scheme with congruence-based completion from \eqref{eq:DGM4} (\texttt{DGM4}) and the Pad\'e-based completion \eqref{eq:Pade01} embedded in the commutator-based splitting method of \cite{moench2026} (\texttt{Split}). The splitting scheme is based on an energy-associated decomposition of the system into a purely dissipative and an energy-conserving subproblem, which are solved using the Pad\'e-based completion $\overline{R}_{0,1}$ and \texttt{DGM4} for $R=0$, respectively. i.e.
\begin{equation*}
\Phi_h = \varphi_{h/6}^{{R}} \circ
\varphi_{h/2}^{{\hat{J}}} \circ
\varphi_{2h/3}^{{R}} \circ
\varphi_{h/2}^{{\hat{J}}} \circ
\varphi_{h/6}^{{R}},
\end{equation*}
where $\varphi^{{R}}$  and $\varphi^{{\hat{J}}}$ denote the numerical flows of
\begin{align*}
\dot{x} &= -R\nabla\mathcal{H}(x),\qquad \quad 
\dot{x} =
\hat J \nabla\mathcal{H}(x), \,\, \hat J=J-\frac{h^2}{72}
\big[
R\nabla^2\mathcal{H}(x),
[R\nabla^2\mathcal{H}(x),J]
\big]=\mathrm{const}.
\end{align*}
The commutator-based skew-symmetric matrix $\hat J$ is here constant despite the state dependence of the Hessian due the underlying diagonal block structure of $R$.  Moreover, $[A,B] = AB-BA$ is the matrix commutator. 

\begin{figure}[tb]
    \centering
    \begin{subfigure}[b]{0.475\textwidth}
      \centering
        \includegraphics{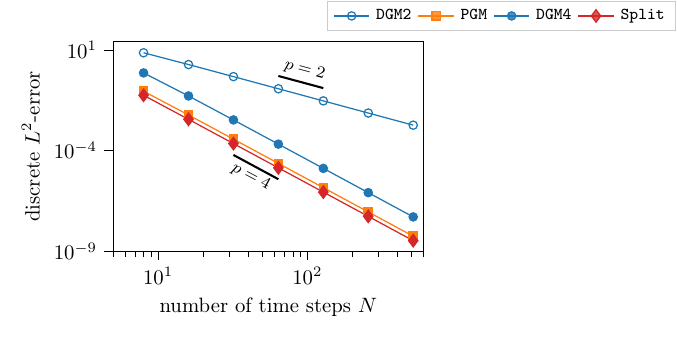}
    \end{subfigure}
    \hfill
    \begin{subfigure}[b]{0.475\textwidth}
        \centering
        \includegraphics{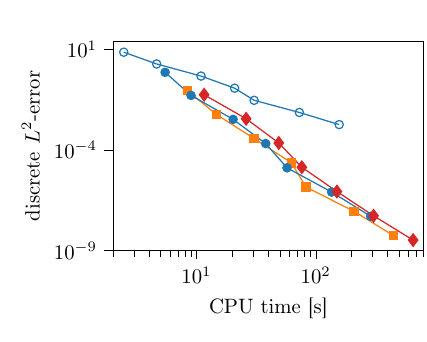}
    \end{subfigure}
    \vspace*{-0.3cm}
    \caption{Dissipative FPU system, $r=0.1$, $\omega=1$, $T=2$ and random initialization $x(0) \sim (\mathcal{U}(0,1))^{4m}$. Discrete $L^2([0,T])$-error over number of time steps $N=T/h$ (left) and over CPU-time (right) for $m=25\,000$.}
    \label{fig:FPU_results}
\end{figure}

The numerical results in Figure~\ref{fig:FPU_results}(left) confirm the theoretical convergence order of the schemes: 
\texttt{DGM4}, \texttt{Split}, and \texttt{PGM} converge with order $p=4$, while the baseline method \texttt{DGM2} is second-order accurate.
In this benchmark, the splitting scheme yields the smallest error constant, whereas \texttt{DGM4} exhibits the largest among the fourth-order methods.
Regarding computational cost, the work-precision diagram in Figure~\ref{fig:FPU_results}(right) shows that the fourth-order discrete gradient schemes are competitive with the established Galerkin method \texttt{PGM}. Moreover all fourth-order schemes outperform \texttt{DGM2} in terms of computational efficiency despite the additional costs of  evaluating the higher-order (correction) terms. 
The work-precision results shown in Figure~\ref{fig:FPU_results}(right) correspond to $m=25\,000$, i.e. a system dimension of $4m=100\,000$. Similar results are obtained for other values of $m$, with the computational cost increasing with the system size.
Since $J$ and $R$ are sparse in this benchmark, the energy-associated decomposition underlying the splitting scheme provides limited structural savings and therefore does not lead to a substantial efficiency advantage over \texttt{DGM4}. 
However, the splitting approach may become more advantageous for systems in which the conservative and dissipative operators exhibit substantially different computational characteristics, for example, when one is sparse and the other dense, when one is constant and the other state-dependent, or when the corresponding dynamics evolve on different time scales.

Taken together, the two benchmarks demonstrate that the modified discrete gradient methods are unconditionally dissipative, reproduce the correct qualitative behavior where the uncompleted correction fails, and achieve fourth-order accuracy at a computational cost competitive with existing fourth-order alternatives.

\section{Conclusion and Outlook}\label{sec:Conclusion}

In this work, we developed fourth-order energy-consistent AVF discrete gradient methods for dissipative Hamiltonian systems. By combining the perturbation principle with a positive semi-definite completion of the dissipation matrix, the resulting methods retain the dissipative structure while achieving fourth-order accuracy. For purely dissipative systems, we developed an alternative exponential completion and structure-preserving rational approximations that can be embedded in energy-associated splitting schemes.
Numerical experiments for the damped physical pendulum and a dissipative Fermi--Pasta--Ulam system confirm the expected convergence orders and energy-consistency and demonstrate the computational competitiveness of the proposed methods relative to established Galerkin approaches.
The introduced concept of congruence-based completion directly enables the systematic construction of higher-order energy-consistent schemes for arbitrary orders ($p>4$). Beyond this direct extension, the framework opens several directions for further research, including its application to other classes of discrete gradients, dissipative Poisson systems with state-dependent structure matrices, and port-Hamiltonian systems with external inputs or algebraic constraints


\end{document}